\documentclass[12pt,a4paper]{amsart}
\usepackage{geometry}
\usepackage{tabls}
\usepackage{url}
\usepackage[utf8]{inputenc}
\usepackage{amsfonts}
\usepackage{amssymb}
\usepackage{mathrsfs}
\usepackage{amsmath}
\usepackage{amsthm}
\usepackage{amscd}
\usepackage{fancyhdr}
\usepackage{enumerate}
\usepackage{paralist}
\usepackage{graphicx}
\usepackage{cite}

\usepackage{tikz}
\usetikzlibrary{decorations.pathreplacing}

\usepackage{footmisc}
\numberwithin{equation}{section}

\theoremstyle{plain}

\newtheorem{Thm}[equation]{Theorem}
\newtheorem{lem}[equation]{Lemma}
\newtheorem{prop}[equation]{Proposition}
\newtheorem{rem}[equation]{Remark}

\newtheorem{conj}[equation]{Conjecture}
\newtheorem{cond}[equation]{Condition}

\begin{document}

\title{ Arithmetic sums and products of infinite  multiple zeta-star values }

\author{Jiangtao Li}

\email{lijiangtao@csu.edu.cn}
\address{Jiangtao Li \\ School of Mathematics and Statistics, HNP-LAMA, Central South University, Hunan Province, China}

\author{Siyu Yang*}

\email{siyuyang@csu.edu.cn}
\address{Siyu Yang \\ School of Mathematics and Statistics, HNP-LAMA, Central South University, Hunan Province, China  }

\begin{abstract}
  Multiple zeta-star values are variants of multiple zeta values which allow equality in the definition. Similar to the theory of continued fractions, every real number greater than $1$ can be naturally realized as a unique infinite multiple zeta-star value. In this paper, we investigate the arithmetic sums and products of infinite multiple zeta-star values with restricted indices. Moreover, inspired by the theory of continued fractions and Cantor set, we propose a series of conjectures concerning the algebraic points and arithmetic sums and products of infinite multiple zeta-star values with certain indices.     
         \end{abstract}

\let\thefootnote\relax\footnotetext{
2020 $\mathnormal{Mathematics} \;\mathnormal{Subject}\;\mathnormal{Classification}$. 11M32,11Y65.\\
$\mathnormal{Keywords:}$  Multiple zeta-star values, Continued fractions. \\
Project funded by the National Natural Science Foundation of China (Grant No. 12571009) and the
Natural Science Foundation of Hunan Province,
China (Grant No. 2026JJ40003).\\
*Corresponding author }

\maketitle

\section{Introduction}

There are multiple zeta values $$ \zeta(k_1, \cdots, k_r) = \sum_{n_1 > \cdots > n_r \geq 1} \frac{1}{n_1^{k_1} \cdots n_r^{k_r}}, \\\ k_1 \geq 2, k_2, \cdots, k_r \geq 1 $$ and multiple zeta-star values $$ \zeta^\star(k_1, \cdots, k_r) = \sum_{n_1 \geq \cdots \geq n_r \geq 1} \frac{1}{n_1^{k_1} \cdots n_r^{k_r}}, \\\ k_1 \geq 2, k_2, \cdots, k_r \geq 1. $$

In this paper, we use the notations from \cite{lit, lid,lir}. For $ \eta $, it is a map from infinite sequences to real numbers: \[
\eta: \mathcal{T}\rightarrow (1,+\infty),
\]
$$  (k_1,\cdots, k_r,\cdots) \mapsto \mathop{\mathrm{lim}}_{r\rightarrow +\infty}\zeta^\star(k_1, \cdots, k_r) =   \sum_{n_1 \geq \cdots \geq n_r \geq \cdots \geq 1} \frac{1}{n_1^{k_1} \cdots n_r^{k_r} \cdots } \; . $$ 
Here  \[
 \mathcal{T}=\Big{\{} (k_1, \cdots, k_r,\cdots)\,\Big{|}\,k_1\geq 2,k_i\geq 1, i\geq 2, k_s\geq 2 \;for \; some \; s\geq 2 \;if \;k_1=2  , k_i\in\mathbb{N}    \Big{\}}.
 \]
 The map $\eta$ is bijective.
 We call the map $\eta$ zeta-star correspondence. 
 The order structure of multiple zeta-star values and the map $\eta$ were discovered by the first author \cite{lit} and Hirose-Murahara-Onozuka \cite{hmo} independently. Additionally, in \cite{hmo},  Hirose, Murahara and Onozuka computed several formulas for the multiple zeta-star values and studied differentiability.  For convenience we write
 \[
 \eta\left( (k_1,\cdots, k_r,\cdots  )  \right)=\zeta^\star(k_1,\cdots,k_r,\cdots).
 \]
 By the zeta-star correspondence, every real number which is greater than $1$ can be viewed as an {\bf unique infinite  multiple zeta-star values}.

What is the set of all infinite multiple zeta-star values whose indices do not exceed $q$? For $p,q\geq 2$, denote $$\mathcal{D}_q =\big{\{} (k_1, \cdots, k_r,\cdots)\in \mathcal{T} \; \big{|} \;  k_i\leq q, i\geq 1  \big{\}},$$
$$     \mathcal{T}_p=\big{\{} (k_1, \cdots, k_r,\cdots)\in \mathcal{T} \; \big{|} \;  k_i\geq p, i\geq 1  \big{\}}.         $$
Denote by $m$  the Lebesgue measure of $\mathbb{R}$.  In \cite{lit}, the  first author proved that:\\
$(i)$ The set $\eta(\mathcal{D}_q )$ is a closed set of $\mathbb{R}$ and $m(  \eta(\mathcal{D}_q )  )=0$, $q\geq 2$;\\
$(ii)$ The Hausdorff dimension of $\eta(      \mathcal{T}_p   )$ is
\[
  \frac{\mathrm{log}\;\alpha_p}{\mathrm{log}\;2}, \forall\,p\geq 2.
 \]
Here for  $p\geq 2$, $\alpha_p      $  is the unique root of the equation 
 $x^{p-1}(x-1)=1$
 which satisfies $\alpha_p\in (1,2)$.

Any real number $x$ can be expressed as a continued fraction: $$ x= u_0 + \frac{1}{u_1 + \frac{1}{ u_2+ \cdots } } := (u_0; u_1, u_2, \cdots), \; u_0\in\mathbb{Z}, u_1,\cdots, u_r,\cdots \in \mathbb{N} . $$ If $x\in \mathbb{Q}$, this expression terminates at finite steps.
Here we call  $u_0$
the integer part of $x$ and $u_1, u_2, \ldots $ the partial quotients of $x$.  For convenience, we denote $$ F(m) = \big{\{ } (u_0; u_1, u_2, \cdots) \; \big{|} \; 1 \leq u_i \leq m, i \geq 1 \big{\}}, $$ 
$$ G(n) = \big{\{}  (u_0; u_1, u_2, \cdots) \; \big{|} \; u_i \geq n, i \geq 1 \big{\}} \cup \bigcup_{r\geq 1} (u_0; u_1,  \cdots, u_r) \; \big{|} \; u_i \geq n, r\geq i \geq 1 \big{\}} .$$
In the definition of $F(m)$, we only allow infinite continued fractions. While in the definition of $G(n)$, finite continued fractions are also allowed.

In 1946, Hall \cite{h} proved that any real number can be expressed as the sum of real numbers whose fractional parts have continued fraction partial quotients bounded by 4, and any number greater than 1 can be expressed as the product of real numbers with the same property. That is, $$ F(4)+F(4)= \mathbb{R}, \\\ F(4) \cdot F(4)\supseteq [1, + \infty). $$ He adopted the idea of the Cantor set to process step by step the partial quotients $ u_1, u_2, \ldots $ of continued fractions, thereby obtaining the target set $F(4)$. Combining this with the conditions under which sumsets of Cantor sets form an interval, he arrived at the conclusion. In Section 3, we partially describe this ingenious technique.

Furthermore, in 1971, Cusick \cite{cus}  proved that $$G(2) + G(2) = \mathbb{R} , \\\ G(2) \cdot G(2) \supseteq [1, + \infty) $$ via the  alternating order structure of continued fractions. Moreover, this is optimal in a certain sense. For $k \geq 3$, neither the set $G(k)$ of numbers whose partial quotients are at least $k$, nor the other set whose partial quotients are at least 2 but bounded above, can cover $  \mathbb{R} $ when summed with itself.

In 1999, Astels \cite{ast} extended Hall's theorem and summarized the 1975 results of James Hlavka \cite{hla}: the equations $$ F(m)+F(n)= \mathbb{R}, \\\ F(m)-F(n)= \mathbb{R} $$ hold if $(m,n)$ equals $(2,5)$ or $(3,4)$. Neither of the above equations holds if $(m,n)$ equals $(2,4).$ Additionally, $ F(3)+F(3) \neq \mathbb{R} $ and $ F(3)-F(3) = \mathbb{R}. $  In addition, Astels employed the thickness to show that for two positive integer sets $B_1$ and $B_2$, if $ t(B_1)t(B_2) \geq 1$, then $$ \epsilon_1 F(B_1) + \epsilon_2 F(B_2) = \mathbb{R}. $$ Here $t(B_1)$ means the thickness of $B_1$, $\epsilon_1, \epsilon_2$ are each one of $\{1,-1\}$, and $ F(B_1) $ means the set of continued fractions with partial quotients belonging to $B_1$ except $u_0$.

One can compare the theory of continued fractions with the theory of zeta-star correspondence. Inspired by the above results of continued fractions, we prove that:
\begin{Thm}\label{low}For $q\geq 2$, we have: \\
$(i)$ 
\[\eta(\mathcal{D}_q)+\eta(\mathcal{D}_q)=[c_q, +\infty).\]
Here 
\[
c_q=2\zeta^\star(\{q\}^{\infty})=   2\prod_{n\geq 2}\frac{1}{1-\frac{1}{n^q}}     .
\]
$(ii)$ 
\[\eta(\mathcal{D}_q)\cdot \eta(\mathcal{D}_q)=[d_q, +\infty).\]
Here 
\[
d_q=\left(\zeta^\star(\{q\}^{\infty})\right)^2=   \prod_{n\geq 2}\frac{1}{\left(1-\frac{1}{n^q}\right)^2}     .
\]
$(iii)$ \[\eta(\mathcal{D}_q)-\eta(\mathcal{D}_q)=\mathbb{R}.\]
$(iv)$ \[\eta(\mathcal{D}_q)/\eta(\mathcal{D}_q)=(0,+\infty).\]\end{Thm}

For $q=2$, one has 
\[
\zeta^\star(\{2\}^{\infty})=   \prod_{n\geq 2}\frac{1}{1-\frac{1}{n^2}} =2.  \]
Thus $c_2=d_2=4$.
Roughly speaking, every real number  $x$ which is bigger than $4$  can be realized as  the sum (product)  of two infinite  multiple zeta-star values of all indices $\leq 2$, 
 i. e. 
 \[
 x=\mathop{\mathrm{lim}}_{r\rightarrow +\infty}\left[ \zeta^\star(k_1,\cdots, k_r)+ \zeta^\star(l_1,\cdots, l_r)\right]
 \]
  \[
 \left(x=\mathop{\mathrm{lim}}_{r\rightarrow +\infty}\left[ \zeta^\star(k_1,\cdots, k_r)\cdot \zeta^\star(l_1,\cdots, l_r)\right]\right)
 \]
 for some $(k_1,\cdots, k_r,\cdots), (l_1,\cdots, l_r,\cdots )\in \mathcal{D}_2$.

 We will review the order structure of multiple zeta-star values in Section 2. After a brief introduction, Hall's approach based on the Cantor  set can be effectively adapted to the case where the target set is the set of binary expressions. That is: 
 \[
 \tau: \widehat{\mathcal{T}} \rightarrow (0,1],
 \]
 $$ (k_1,\cdots, k_r,\cdots)\mapsto \frac{1}{2^{k_1}}+\frac{1}{2^{k_1+k_2}}+\cdots +\frac{1}{2^{k_1+\cdots+k_r}}+\cdots .$$ Here 
 \[
  \widehat{\mathcal{T}}=\big{\{} (k_1,\cdots, k_r,\cdots)\;\big{|}\; k_r\in \mathbb{N}, k_r\geq 1, r\geq 1   \big{\}}.
 \]
  As the maps $\eta$ and $\tau$ have the same order structure, following the same idea, we subdivide and compute the infinite  multiple zeta-star values.  The product case can be derived from the sum case by taking logarithms. 
  
In fact, from \cite{new} we have Newhouse's thickness theorem. This is a more general method for obtaining sum sets of Cantor sets. Indeed, one might have used it to prove Theorem 1.1 as well. Can a similar result be obtained for the set of infinite multiple zeta-star values with $k_i$ bounded below by applying this theorem? Unfortunately, we can prove that:
  \begin{Thm}\label{grea}For $p\geq 2$, define 
\[
 \overline{\eta(\mathcal{T}_p)} 
\]
is the closure of $\eta(\mathcal{T}_p)$ in $\mathbb{R}$. 
In fact it is easy to check that
\[
 \overline{\eta(\mathcal{T}_p)} =\eta(\mathcal{T}_p) \bigcup \big{\{} \zeta^\star(k_1,\cdots, k_r)\in \mathcal{T}\;\big{|}\;    k_i\geq p, r\geq 1     \big{\}} \bigcup \big{\{}1\big{\}}.
\]
$(i)$ 
\[
\overline{\eta(\mathcal{T}_p)}+\overline{\eta(\mathcal{T}_p)}  \]
is not a closed interval for any $p\geq 2$.\\
$(ii)$ 
\[
  \overline{\eta(\mathcal{T}_p)}\cdot \overline{\eta(\mathcal{T}_p)}  \]
  is not a closed interval for any $p\geq 2$.\\
  
  \end{Thm}
  
Theorem \ref{low} and Theorem \ref{grea} provide partial results for the questions proposed by the first author in Remark $4.4$, \cite{lit}.
 Although Theorem \ref{low} is similar to the case of continued fractions, Theorem \ref{grea} is quite different from the case of continued fractions.
The Newhouse Thickness Theorem states that the arithmetic sum 
$$ C_1+C_2= \{x_1+x_2 \; \big{|} \; x_1 \in C_1,x_2 \in C_2 \} $$ is an interval if $ t(C_1) \cdot t(C_2) \geq 1. $ Here the Cantor sets satisfy that the maximum gap of $C_1$ does not exceed the diameter of $C_2$ and conversely. One can apply the Newhouse Thickness theorem to show that 
\[
\overline{\tau(\mathcal{T}_2)}+\overline{\tau(\mathcal{T}_2)}  
\]
is a closed interval. Here $\overline{\tau(\mathcal{T}_2)}$ is the closure of ${\tau(\mathcal{T}_2)} $.
Although the maps $\eta$ and $\tau$ have the same order structure, the two Cantor sets 
$\overline{\eta(\mathcal{T}_p)}$  and  $\overline{\tau(\mathcal{T}_p)}$  behave very differently.

In the theory of continued fractions, Mercat \cite{mer} observed that there exist infinitely many real quadratic fields $\mathbb{Q}[\sqrt{\delta}]$ admitting infinitely many periodic continued fractions with partial quotients $u_i$ bounded above by 2 (i.e. $u_i \in \{1,2 \}$).   McMullen \cite{mcm} conjectures that in any real quadratic field $\mathbb{Q}(\sqrt{\delta})$, there exist infinitely many periodic continued fractions whose partial quotients take only the integer 1 and 2.  

 For the theory of multiple zeta-star values, after extensive calculations, we believe that the following conjecture is true.
  \begin{conj}
  Denote by $\overline{\mathbb{Q}}$ the set of algebraic numbers in $\mathbb{C}$. Then
 \[ \eta(\mathcal{D}_2)\bigcap\; \overline{\mathbb{Q}} \;=\big{\{}  n\;\big{|}\; n\in \mathbb{Z}, n\geq 2             \big{\}}  .\]
  \end{conj}
    Furthermore, one may have:
  \begin{conj}
	Denote by $\overline{\mathbb{Q}}$ the set of algebraic numbers in $\mathbb{C}$. Then
	\[ \eta(\mathcal{D}_q)\bigcap\; (\overline{\mathbb{Q}} \setminus \mathbb{Z} )\;=  \varnothing  ,\;\; q\geq 2.\]
  \end{conj}

  Based on a thorough analysis of Theorem 1.2, we conjecture:
  \begin{conj} For $p\geq 2$ and any $a,b \in \mathbb{R}$, one has
  \[
  (a, b) \nsubseteq \overline{\eta(\mathcal{T}_p)}+\overline{\eta(\mathcal{T}_p)} . \]
  In other words, the set \[\overline{\eta(\mathcal{T}_p)}+\overline{\eta(\mathcal{T}_p)}\]
  does not have any interior points for any $p\geq 2$.   
  \end{conj}

There is a Palis conjecture (1983) which was formally proposed in \cite{pal}. Specifically, Palis conjectured that for generic one-dimensional regular Cantor sets $C_1, C_2$, if the sum of their Hausdorff dimensions exceeds 1, then their arithmetic difference $$C_1 - C_2 = \{x - y : x \in C_1,\; y \in C_2\} $$ is either a set of Lebesgue measure zero or contains an interval. Inspired by Palis Conjecture and Conjecture 1.5, we propose the following conjecture:
  \begin{conj} Let $m(A)$ denotes the Lebesgue measure of the set $A$, one has \[m  (\; \overline{\eta(\mathcal{T}_2)}+\overline{\eta(\mathcal{T}_2)} \;) = 0.\]   
  \end{conj}
While the Palis conjecture has not yet been completely proved, most of its consequences have already been achieved. In 2001, Moreira and Yoccoz \cite{may} proved Palis' conjecture for non-linear regular Cantor sets: for generic pairs with sum of the Hausdorff dimensions $\operatorname{HD}(K_1)+\operatorname{HD}(K_2)>1$, the arithmetic difference $K_1-K_2$ contains an interval.  In the subclass of homogeneous Cantor sets (with equal contraction ratios and disjoint initial intervals), Takahashi \cite{tak} (2019) proved that if $\operatorname{HD}(K_1)+\operatorname{HD}(K_2)>1$ and the uniform self-similar measures have $L^2$-density, then one can perturb $K_1$ arbitrarily slightly (within the homogeneous class) so that $K_1+K_2$ contains an interval. This provides an affirmative answer to a weaker form of the Palis conjecture in the affine setting. So regarding the Palis conjecture, it is known that the case of nonlinear regular Cantor sets has been fully proved, while the classical affine (self-similar) Cantor set case remains to be fully resolved.

\section{The order structure of multiple zeta-star values}
This section will review the total order structure on the set of multiple zeta-star values. Some special infinite multiple zeta-star values are calculated. The main references are \cite {lid,lit}.

Denote by
\[
S = \{ (k_1, \cdots, k_r) \; \big{|} \; k_1 \geq 2,\; k_2, \cdots, k_r \geq 1, \; r \geq 1 \}.
\]

Define an order \(\succ\) on \(S\) by the following two rules. \\
Rule 1:
\[
(k_1, \cdots, k_r, k_{r+1}) \succ (k_1, \cdots, k_r)
\]
for any $ (k_1, \cdots, k_r, k_{r+1}) \in S.$\\
Rule 2:
\[
(k_1, \cdots, k_r) \succ (m_1, \cdots, m_s)
\]
if the first $i$ terms are equal ( $k_j = m_j$ for $1 \leq j \leq i$ ) while the ($i+1$)-th term differs ( $k_{i+1} < m_{i+1}$ for some $i \geq 0$ ). 

\begin{lem}\label{1inf}
For $(k_1,\cdots, k_r)\in\mathcal{S}$ and $m\geq 1$, one has 
\[
\zeta^\star(k_1,\cdots,k_{r-1}, k_r+1,\{1\}^m)< \zeta^\star(k_1,\cdots,k_{r-1},k_r).
\]
\end{lem}
\noindent{\bf Proof:} We have
\[
\begin{split}
&\;\;\;\; \zeta^\star(k_1,\cdots,k_{r-1}, k_r+1,\{1\}^m)            \\
&=\sum_{n_1\geq \cdots\geq n_r\geq \cdots \geq n_{r+m}\geq 1} \frac{1}{n_1^{k_1}\cdots n_{r-1}^{k_{r-1}} n_r^{k_r+1} n_{r+1}\cdots n_{r+m}}\\
&=\sum_{n_1\geq \cdots\geq n_r\geq  1}\frac{1}{n_1^{k_1}\cdots n_{r-1}^{k_{r-1}} n_r^{k_r}} \cdot \frac{1}{n_r}\left(  \sum_{n_r\geq n_{r+1}\geq \cdots \geq n_{r+m}\geq 1}   \frac{1}{n_{r+1}\cdots n_{r+m}}    \right).
 \end{split}
\]
For $n\geq 2,m\geq 1$, it is clear that
\[
\frac{1}{n}\sum_{n\geq n_1\geq \cdots \geq n_m\geq 1}\frac{1}{n_1\cdots n_m}<\frac{1}{n}\prod_{n_1=2}^n\left(1+\frac{1}{n_1}+\cdots+\frac{1}{n_1^r}+\cdots       \right)=1.
\]
Thus 
\[
\zeta^\star(k_1,\cdots,k_{r-1}, k_r+1,\{1\}^m)< \zeta^\star(k_1,\cdots,k_{r-1},k_r).\]
$\hfill\Box$\\

\begin{prop}\label{ord}
If $(k_1,\cdots, k_r)\succ (l_1,\cdots, l_s)$ for $(k_1,\cdots,k_r),(l_1,\cdots,l_s)\in\mathcal{S}$, then 
\[
\zeta^\star( k_1,\cdots, k_r   )>\zeta^\star(    l_1,\cdots, l_s    ).
\]
\end{prop}
\noindent{\bf Proof:}
If $r>s$, then we have 
\[
(k_1,\cdots,k_s)=(l_1,\cdots, l_s)
\]
or 
\[
(k_1,\cdots, k_s)\succ (l_1,\cdots, l_s).
\]
$(i)$ For $(k_1,\cdots,k_s)=(l_1,\cdots, l_s)$, one has
\[
\begin{split}
&\;\;\;\; \zeta^\star(k_1,\cdots, k_r)      \\
&= \sum_{n_1\geq \cdots\geq n_s\geq n_{s+1}\geq \cdots \geq n_r\geq 1}  \frac{1}{n_1^{k_1}\cdots n_s^{k_s}n_{s+1}^{k_{s+1}}\cdots n_r^{k_r}}            \\
&= \left(  \sum_{\substack{n_1\geq \cdots\geq n_s\geq n_{s+1}\geq \cdots \geq n_r\geq 1\\ n_{s+1}=1}}    +  \sum_{\substack{n_1\geq \cdots\geq n_s\geq n_{s+1}\geq \cdots \geq n_r\geq 1\\ n_{s+1}>1}}   \right)       \frac{1}{n_1^{k_1}\cdots n_s^{k_s}n_{s+1}^{k_{s+1}}\cdots n_r^{k_r}}                                \\
&>\sum_{n_1\geq \cdots\geq n_s\geq  1}  \frac{1}{n_1^{k_1}\cdots n_s^{k_s}} =\zeta^\star(l_1,\cdots,l_s).        \end{split}
\]
$(ii)$For $(k_1,\cdots,k_s)\succ (l_1,\cdots, l_s)$, one has
$(k_1,\cdots,k_{i-1})=(l_1,\cdots,l_{i-1}), k_{i}+1\leq l_{i}$ for some $i\leq s$.
\[
\begin{split}
& \;\;\;\; \zeta^\star(l_1,\cdots, l_s)        \\
&= \sum_{n_1\geq \cdots\geq n_s\geq 1}\frac{1}{n_1^{l_1}\cdots n_s^{l_s}}                      \\
&\leq \sum_{n_1\geq \cdots \geq n_s\geq 1}\frac{1}{n_1^{k_1}\cdots n_{i-1}^{k_{i-1}} n_{i}^{k_{i}+1} n_{i+1}\cdots n_s}\\
&< \zeta^\star(k_1,\cdots,k_{i-1}, k_i).\\
\end{split}
\]
Here the second inequality follows from Lemma \ref{1inf}. By $(i)$, one has
\[
\zeta^\star(l_1,\cdots, l_s)<\zeta^\star(k_1,\cdots,k_r).\]
If $r\leq s$, then we have $(k_1,\cdots,k_{i-1})=(l_1,\cdots,l_{i-1}), k_{i}+1\leq l_{i}$ for some $i\leq r$.
 By the same argument as $(ii)$, one has 
\[
\zeta^\star(l_1,\cdots, l_s)<\zeta^\star(k_1,\cdots,k_r).\]
$\hfill\Box$\\

Next, we compute some special infinite multiple zeta-star values for certain sequences. There are: 
\[
\begin{aligned}
&\;\; \zeta^\star ( (k_1,k_2, \ldots, k_r, \{q\}^\infty) )
\\
=&\,\eta\left( (k_1,k_2,\cdots,k_r,\{q\}^{\infty})      \right)   
\\= &\sum_{n_1 \geq n_2 \geq \cdots \geq 1} \frac{1}{n_1^{k_1} \cdots n_r^{k_r} n_{r+1}^q n_{r+2}^q \cdots } 
\\= &\sum_{n_1 \geq \cdots \geq n_r \geq 1} \frac{1}{n_1^{k_1} \cdots n_r^{k_r}} \cdot \sum_{n_r\geq n_{r+1} \geq n_{r+2} \cdots \geq 1} \frac{1}{n_{r+1}^q n_{r+2}^q \cdots } 
\\= & \sum_{n_1 \geq \cdots \geq n_r \geq 1} \frac{1}{n_1^{k_1} \cdots n_r^{k_r}} \prod_{n=2}^{n_{r}} \left(1 + \frac{1}{n^q} + \frac{1}{n^{2q}} + \cdots\right) 
\\= &\sum_{n_1 \geq \cdots \geq n_r \geq 1} \frac{1}{n_1^{k_1} \cdots n_r^{k_r}} \prod_{n=2}^{n_{r}} \left(1 - \frac{1}{n^q}\right)^{-1}.
\end{aligned}
\]
Here we interpret that $\prod_{n=2}^{n_r} \left(1-\frac{1}{n^q} \right)^{-1}=1 $ for $n_r=1$. Additionally, \[\begin{aligned} q=1, \; &\prod_{n=2}^{n_{r}} \left(1 - \frac{1}{n^q}\right)^{-1} = n_r \\ q=2, \; &\prod_{n=2}^{n_{r}} \left(1 - \frac{1}{n^q}\right)^{-1} = \prod_{n=2}^{n_{r}} \left(\frac{(n-1)(n+1)}{n^2}\right)^{-1} = \frac{2n_r}{n_r+1}. \end{aligned}\]  

\begin{rem}
In fact, by the total order structure of multiple zeta-star values, one can show that the zeta-star correspondence map \[
\eta: \mathcal{T}\rightarrow (1, +\infty),
\]
\[
(k_1,\cdots,k_r,\cdots)\mapsto \mathop{\mathrm{lim}}_{r\rightarrow +\infty} \zeta^\star(k_1, \cdots,k_r)
\]
is injective. To show that the map is surjective, one need to show that some multiple series tend to 
 $0$ as $r\rightarrow +\infty$. See \cite{lit} for more details.
\end{rem}

\section{ Arithmetic sums of binary expression    }
In this section, by using the theory of Hall  about the sum of Cantor set, we discuss the arithmetic sums of binary expression. The main reference is Hall \cite{h} .

Denote by 
\[
\widehat{\mathcal{T}}=\bigg{\{}(k_1,\cdots, k_r,\cdots)\;\bigg{|}\; k_r\in \mathbb{N}, k_r\geq 1, r\geq 1\bigg{\}}.
\]
There is a natural bijection
\[
\tau:  \widehat{\mathcal{T}}\rightarrow (0,1],
\]
\[
(k_1,\cdots, k_r,\cdots)\mapsto \frac{1}{2^{k_1}}+\frac{1}{2^{k_1+k_2}}+\cdots +\frac{1}{2^{k_1+\cdots+k_r}}+\cdots .
\]
Define 
\[
B_m=\bigg{\{} (a_n)_{n\geq1}\in \widehat{\mathcal{T}} \;\bigg{|}\; a_n\in\mathbb{N}, 1\leq a_n\leq m\bigg{\}}.
\]
\[
L_m=\bigg{\{} (a_n)_{n\geq1}\in \widehat{\mathcal{T}} \;\bigg{|}\; a_n\in\mathbb{N}, a_n\geq m\bigg{\}}.
\]
Is there some $m\geq 2$ such that 
\[
(c_1, 2) \subseteq \tau(B_m)+\tau(B_m) 
\]
for some $0<c_1<1$? Similarly, is there any $m \geq 2$ such that
\[
(c_2,c_3) \subseteq \tau(L_m)+\tau(L_m) 
\]
for some $ 0< c_2 < c_3 < 1 $?

For $m=2$, we have an interval with a minimum of $ \tau(B_2) $: $$\tau (2,2,2,\ldots)=\frac{1}{3}$$ and a maximum of $\tau(B_2)$: $$\tau (1,1,1,\ldots)=1.$$ By subdividing $[\frac{1}{3},1]$, we can derive a Cantor set that is equal to $\tau (B_2)$. 

For any closed interval $A=[x,x+a]$, we can remove an open subinterval from $A$ write $ A_{12}=(x+a_1,x+a_1+a_{12}) $, leaving two closed intervals $ A_1=[x,x+a_1] $ and $ A_2=[x+a_1+a_{12} ,x+a_1+a_{12}+a_2] $ where $ a=a_1+a_{12}+a_2 $. At each stage, we apply the same subdivision to every remaining interval $A_1$, $A_2$, $\cdots$. The set of points in $A$ that are never removed forms the Cantor set $L(A)$. 

Obviously, $L(A)$ is a perfect set. We also know that $L(A)$ either contains an interval or is nowhere dense. Its Lebesgue measure may be zero or any positive number less than the length of $A$. The precise nature of $L(A)$ depends on the ratios of the lengths of subintervals such as $a_1$, $a_{12}$ and $a_2$.

For any two intervals $ A=[x,x+a] $ and $ B=[y,y+b] $,  define the sum of $A$ and $B$ as the interval $$ A+B=[x+y,x+y+a+b].$$ If $ \alpha \in A $, $ \beta \in B $, write $\gamma = \alpha +\beta $, then $ \gamma \in A+B $ and conversely if we have $ \gamma \in A+B $ there exist $ \alpha \in A $ and $ \beta \in B $ with $ \gamma =\alpha+\beta $. 

Now we can consider the properties of the ``Vector sums'' $ L(A)+L(B) $.  When the sum of Cantor set is the whole interval $A+B$? Let $C$ be any subdividing interval of $A$ with $ C=[u,u+c] $ from which the gap $ C_{12}=(u+c_1,u+c_1+c_{12} )$  is deleted, leaving $ C_{1}=[u,u+c_1] $ and $ C_{2}=[u+c_1+c_{12},u+c_1+c_{12}+c_2] $, where $ c=c_1+c_{12}+c_2 $. (Let $D$ be any subdividing intervals of $B$ with $ D=[v,v+d] $, constructed in the same way as $C$. ) Then by considering the sum $C+C=[2u,2u+2c]$ we can get
\begin{align*}
C_1+C_1 &=[2u,2u+2c_1]\\
C_1+C_2 &=C_2+C_1=[2u+c_1+c_{12},2u+2c_1+c_{12}+c_2]\\
C_2+C_2 &=[2u+2c_1+2c_{12},2u+2c].
\end{align*}

 By comparing the endpoints of the intervals, we easily find that if $ c_{12}>c_1 $ there is a gap between $ C_1+C_1 $ and $ C_1+C_2 $,  while if $ c_{12}>c_2 $ there is a gap between $ C_1+C_2 $ and $ C_2+C_2 $. To avoid the appearance of gaps, we obtain a necessary condition for the sum of Cantor set $ L(A)+L(B) $ to cover the interval $ A+B $. 

\begin{cond}\label{con}(Hall)
The length of $C_{12}$ that is deleted by a subinterval C should not be longer than the length of the interval remained $C_1$ or $C_2$. It is $ c_{12} \leq c_1 $ and $ c_{12} \leq c_2 $. 
\end{cond}

 But this is not enough. Suppose that there are an interval $A$ and an interval $B$ that is divided into three equal parts in the Cantor set construction, denoted as 
 \[
 A=[x,x+a] \quad B=[y,y+b]
 \]
\[
B_1=[y,y+b/3] \quad B_2=[y+2b/3,y+b].
\]
Here $ b_1=b_{12}=b_2 $ and Condition \ref{con} is satisfied. We have
\[
A+B_1=[x+y,x+y+a+b/3]
\]
\[
A+B_2=[x+y+2b/3,x+y+a+b].
\]
Then $ L(A)+L(B) $ is covered by the intervals $A+B_1$ and $A+B_2$. But if $ a<b/3 $ there is a gap between the two intervals 
$A+B_1$ and $A+B_2$. By a slight generalization of this conclusion, we  see that $ \frac{1}{3} \leq a/b \leq 3 $ is also a necessary condition. Combining these two processes, we arrive at the following important theorem. 

\begin{Thm}\label{asu}(Hall)  If the subdivisions of intervals $ A=[x,x+a] $ and $ B=[y,y+b] $ satisfy Condition \ref{con} at all stages and $ \frac{1}{3} \leq a/b \leq 3 $, then $ L(A)+L(B) $ is the entire interval $A+B$. 
\end{Thm}

This theorem is proved in Hall \cite{h} by a sequence of intervals with monotonically decreasing lengths.

Denote by $B:=[\frac{1}{3},1]$. Now we construct the Cantor set $L(B)$ by subdividing $B$. The subintervals we shall use are of three types: 
\[
T :\quad (a_1,a_2,a_3,\ldots,a_r,a_{r+1},\ldots), a_1=k_1, \ldots, a_r=k_r,1 \leq a_j \leq 2, j \geq r + 1
\]
\[
T_1: \quad (a_1,a_2,a_3,\ldots,a_r,a_{r+1},\ldots),  a_1=k_1,\ldots,a_r=k_r, a_{r+1}=1, 1 \leq a_j \leq 2, j \geq r + 2
\]
\[
T_2 :\quad (a_1,a_2,a_3,\ldots,a_r,a_{r+1},\ldots),  a_1=k_1,\ldots,a_r=k_r, a_{r+1} =2,1 \leq a_j \leq 2, j \geq r + 2.
\]
Here $T$ represents the closed interval spanned by the maximum and minimum value of
$$\tau(a_1,a_2,a_3,\ldots,a_r,a_{r+1},\ldots).$$ The same applies to $T_1$ and $T_2$. Here  $k_1,k_2,\ldots,k_r$ are fixed and equal to 1 or 2.

Following the approach in the paper,we can divide $T$ to two closed subintervals separated $ T_1 $ and $ T_2 $ by a gap. Then by repeating the process, we get a Cantor set $L(B)$.

Given a number of $\tau(B_2)$, we may at any stage of subdivision find a subinterval that contains it.
There is a bijection map $\tau$. Thus, any number not in $\tau(B_2)$ will be represented by a from in $T$ with some $a_i$ exceeding 2. If this occurs as the $a_r$, it would not be included in $L(B)$. Hence, the Cantor point set $L(B)$ is the set of numbers $\tau(B_2)$.

Recall Theorem 2.2, if we prove that the subdivision of $B$ satisfies Condition \ref{con} at any stages, we can get the sum of Cantor set $L(B)+L(B)$ is the interval $ B+B=[\frac{2}{3},2] $. Let us divide $T$ to $T_1$ and $T_2$, and then calculate their length. Denote the maximal and minimum values of $T$ by M and N, respectively.
\[
T :N= \tau(k_1,k_2,\ldots,k_r,2,2,\ldots)=L+\frac{1}{2^{k_1+\ldots+k_r+2}}+\frac{1}{2^{k_1+\ldots+k_r+2+2}}+\ldots
\]
with $ L=\frac{1}{2^{k_1}} + \ldots + \frac{1}{2^{k_1+\ldots+k_r}} $.
\[
M= \tau(k_1,k_2,\ldots,k_r,1,1,\ldots)=L+\frac{1}{2^{k_1+\ldots+k_r+1}}+\frac{1}{2^{k_1+\ldots+k_r+1+1}}+\ldots.
\]
Then the length of $T$ at the stage $(k_1,\cdots, k_r)$ is 
\begin{align*}
M-N &=\frac{1}{2^{k_1+\ldots+k_r}} \left[ \left( \frac{1}{2}+\frac{1}{2^{2}}+\ldots \right)  -\left(\frac{1}{2^{2}}+\frac{1}{2^{4}}+\ldots \right) \right]\\
&=\frac{1}{2^{k_1+\ldots+k_r}}(1-\frac{1}{3})\\
&=\frac{2}{3}K ,
\end{align*}
where  $K =\frac{1}{2^{k_1+\ldots+k_r}}$.
\[
T_1 :N= \tau(k_1,k_2,\ldots,k_r,1,2,2,\ldots)=L+\frac{1}{2^{k_1+\ldots+k_r+1}}+\frac{1}{2^{k_1+\ldots+k_r+1+2}}+\ldots
\]
\[
M= \tau(k_1,k_2,\ldots,k_r,1,1,\ldots)=L+\frac{1}{2^{k_1+\ldots+k_r+1}}+\frac{1}{2^{k_1+\ldots+k_r+1+1}}+\ldots.
\]
Then the length of $T_1$ is 
\begin{align*}
M-N &=\frac{1}{2^{k_1+\ldots+k_r}} \left[ \left( \frac{1}{2}+\frac{1}{2^{2}}+\ldots \right) - \left(\frac{1}{2}+\frac{1}{2^3}+\ldots \right) \right]\\
&=\frac{1}{2^{k_1+\ldots+k_r}}(1-\frac{2}{3})\\
&=\frac{1}{3}K .
\end{align*}
\[
T_2 :N= \tau(k_1,k_2,\ldots,k_r,2,2,\ldots)=L+\frac{1}{2^{k_1+\ldots+k_r+2}}+\frac{1}{2^{k_1+\ldots+k_r+2+2}}+\ldots
\]
\[
M= \tau(k_1,k_2,\ldots,k_r,2,1,1,\ldots)=L+\frac{1}{2^{k_1+\ldots+k_r+2}}+\frac{1}{2^{k_1+\ldots+k_r+2+1}}+\ldots.
\]
Then the length of $T_2$ is 
\begin{align*}
M-N &=\frac{1}{2^{k_1+\ldots+k_r}} \left[ \left( \frac{1}{2^{2}}+\frac{1}{2^{3}}+\ldots \right) - \left(\frac{1}{2^{2}}+\frac{1}{2^{4}}+\ldots \right) \right]\\
&=\frac{1}{2^{k_1+\ldots+k_r}}(\frac{1}{2} - \frac{1}{3})\\
&=\frac{1}{6}K .
\end{align*}

Thus for the length of the gap $ T_{12}=(\frac{2}{3}- \frac{1}{3}- \frac{1}{6})K= \frac{1}{6}K $, we have $ T_{12} \leq T_1 $ and $ T_{12} \leq T_2 $ that satisfy Condition \ref{con} at any stages. So we have
\begin{prop}
$\tau(B_2)+\tau(B_2)=[\frac{2}{3}, 2]$.
\end{prop}

Moreover, by taking $m=k$ for an arbitrary $k \geq 2$ we can prove that 
\[
(c_1, 2) \subseteq \tau(B_m)+\tau(B_m) 
\]
for some $0<c_1<1$, and moreover $0<c_1 \leq\frac{2}{3}$ can be achieved. 

From $$ \tau (1,1,\ldots)=1,$$ $$ \tau(k,k,\ldots)=\frac{1}{2^k}+\frac{1}{2^{k+k} }+\ldots=\frac{1}{2^k-1} $$ and  the above analysis, we can obtain $ \tau(B_k) $ as a Cantor point set $ L(A) $ by subdividing $A:=  [ \frac{1}{2^k-1}, 1]$. The specific subdivision of subintervals is as follows types: 
 
\[
T_1:(a_1,a_2,a_3,\ldots,a_r,a_{r+1},\ldots),  a_1=k_1, \ldots, a_r=k_r, a_j \in [1,k] , j \geq r + 1
\]
\[
T_2: (a_1,a_2,a_3,\ldots,a_r,a_{r+1},\ldots), a_1=k_1,\ldots,a_r=k_r, a_{r+1} \in [2,k] , a_j \in [1,k], j \geq r + 2
\]
\[
T_3: (a_1,a_2,a_3,\ldots,a_r,a_{r+1},\ldots), a_1=k_1,\ldots,a_r=k_r, a_{r+1} \in [3,k], a_j \in [1,k], j \geq r + 2.
\]
\[
\ldots
\]
\[
T_{k-1}:  (a_1,a_2,a_3,\ldots,a_r,a_{r+1},\ldots), a_1=k_1,\ldots,a_r=k_r, a_{r+1} \in [k-1,k], a_j \in [1,k], j \geq r + 2.
\]
Here $A$ is of type with $r=0$ and  $T_i$ represents the interval spanned by the maximum and minimum value of
$$\tau(a_1,a_2,a_3,\ldots,a_r,a_{r+1},\ldots).$$ 
In the above definition, $k_j \in [1,k], 1 \leq j \leq r $. Now we can subdivide any $T_i$ with $1 \leq i \leq k-1$ into two types mentioned above of closed intervals and one gap. For example, an interval of type $T_1$ may be divided to an interval of type $T_1$ with $$a_{r+1}=1,1 \leq a_j \leq k, j \geq r+2$$ and an interval of type $T_2$ with $$2 \leq a_{r+1} \leq k,1 \leq a_j \leq k, j \geq r+2,$$ and the gap between these two intervals is removed. An interval of type $T_2$ may be divided to an interval of type $T_1$ with $$a_{r+1}=2,1 \leq a_j \leq k, j \geq r+2$$ and an interval of type $T_3$ with $$3 \leq a_{r+1} \leq k,1 \leq a_j \leq k, j \geq r+2,$$ and the gap between these two intervals iszz removed. So an interval of type $T_i$ with arbitrary $1\leq i \leq k-2$ may be divided to an interval of type $T_1$ with $$a_{r+1}=i,1 \leq a_j \leq k, j \geq r+2$$ and an interval of type $T_{i+1}$ with $$i+1 \leq a_{r+1} \leq k,1 \leq a_j \leq k, j \geq r+2,$$ and the gap between these two intervals is removed. Finally an interval of type $T_{k-1}$ may be divided to an interval of type $T_1$ with $$a_{r+1}=k-1,1 \leq a_j \leq k, j \geq r+2$$ and an interval of type $T_1$ with $$ a_{r+1}=k,1 \leq a_j \leq k, j \geq r+2,$$ and the gap between these two intervals is removed.  

Similarly, we arrive at a conclusion for $m=k$, which asserts that $L(A)$ constructed in this way coincides with the entire set $ \tau(B_k) $. To establish $ \tau(B_k)+\tau(B_k)=L(A)+L(A)=A+A $, it suffices to prove, through computation, that the length of the gap in subdivision of $T_k$ is no greater than the lengths of its two respective retained closed intervals. 

If $1 \leq i \leq k-1$, we have
\[
T_i: N= \tau(k_1,k_2, \ldots, k_r, k, k, \ldots),M=\tau(k_1,k_2, \ldots, k_r, i, 1, 1, \ldots). 
\]
Then the length of $T_i$ is
\[
M-N=K\left[ \left(\frac{1}{2^i}+ \frac{1}{2^{i+1}} + \frac{1}{2^{i+2}}+ \ldots)- (\frac{1}{2^k}+\frac{1}{2^{2k}}+ \ldots\right) \right]=K\left( \frac{1}{2^{i-1}} -\frac{1}{2^k-1} \right),
\]
where $K$ is $\frac{1}{2^{k_1+\ldots+k_r}}$.
\[
Subinterval \quad T_1:N= \tau(k_1,k_2, \ldots, k_r, i, k, k, \ldots),M=\tau(k_1,k_2, \ldots, k_r, i, 1, 1, \ldots).
\]
Then the length of subinterval $T_1$ is
\[
M-N=K\left[ \left(\frac{1}{2^i}+ \frac{1}{2^{i+1}} + \frac{1}{2^{i+2}}+ \ldots\right) -\left( \frac{1}{2^i}+ \frac{1}{2^{i+k}}+ \frac{1}{2^{i+2k}}+ \ldots\right)\right]=K\left(\frac{1}{2^{i-1}}-\frac{2^{k-i}}{2^k-1} \right).
\]
\[
Subinterval \quad T_{i+1}:N= \tau(k_1,k_2, \ldots, k_r, k, k, \ldots),M=\tau(k_1,k_2, \ldots, k_r, i+1, 1, 1, \ldots).
\]
Then the length of $T_i$ is
\[
M-N=K\left[\left( \frac{1}{2^{i+1}} + \frac{1}{2^{i+2}}+ \ldots\right)- \left(\frac{1}{2^k}+\frac{1}{2^{2k}}+ \ldots\right) \right]=K \left( \frac{1}{2^i}- \frac{1}{2^k-1} \right).
\]
Hence, the length of gap is 
\[
K\left[ \left( \frac{1}{2^{i-1}} -\frac{1}{2^k-1}\right) -\left(\frac{1}{2^{i-1}}-\frac{2^{k-i}}{2^k-1}\right)- \left( \frac{1}{2^i}- \frac{1}{2^k-1}\right)\right]= K\left( \frac{2^{k-i}}{2^k-1} -\frac{1}{2^i} \right)= \frac{K}{2^i(2^k-1)}.
\]
It is straightforward to verify that when $1 \leq i \leq k-1$, $$K \frac{1}{2^i(2^k-1)} \leq K\left(\frac{1}{2^{i-1}}-\frac{2^{k-i}}{2^k-1}\right), \; K \frac{1}{2^i(2^k-1)} \leq  K\left( \frac{1}{2^i}- \frac{1}{2^k-1}\right).$$

Through the above process, we have proven that the subdivision of A satisfies Condition \ref{con}. It then follows from Theorem \ref{asu} that we can obtain the sum of Cantor set $L(A)+L(A)$ is the interval $ A+A=[\frac{2}{2^k-1}, 2] $. 
\begin{prop}
For $k\geq 3$,
\[
\tau(B_k)+\tau(B_k) =A+A=\left[\frac{2}{2^k-1}, 2\right] .\]
\end{prop}

\section{Sums and products of infinite depth multiple zeta-star values}

Theorem 1.1 can now be proved using an argument similar to that presented in Section 3. For $ \mathcal{D}_q=\big{\{} (k_1,\cdots, k_r,\cdots)\in \mathcal{T}\;\big{|}\;  k_i\leq q, i\geq 1 \big{\}} $ with $ q \geq 2 $ and $k_1\geq 2$, we have $$\eta(\mathcal{D}_q)= [ \zeta^\star(\{q\}^{\infty}), +\infty) :=D . $$ Since $\zeta^\star(2,\{1\}^{\infty})$ is divergent here, we regard it as positive infinity, and treat the length of $D$ as positive infinity. It is convenient to set $$ D= [ \zeta^\star(\{q\}^{\infty}), \zeta^\star(2,\{1\}^{\infty}) ) .$$ 

By subdividing $ \eta(\mathcal{D}_q) $ in the same way as $ \tau(B_k) $, we obtain the Cantor point set $L(D)$ and then apply Theorem 2.2 to yield the result $$ D+D=[c_q,+\infty) $$ with $ c_q=2\zeta^\star(\{q\}^{\infty}) $. This allows us to view the rightmost endpoint of the Cantor set $L(D)$ as positive infinity, since $D$ is a half-line on the real axis. Consequently, the rightmost subinterval in the subdivision is an infinite closed interval with infinite length. Then the similar subdivision of subintervals is as follows types ($k_1\geq 2$):
\[
T_1:(a_1,a_2,a_3,\cdots,a_r,a_{r+1},\cdots),  a_1=k_1, \ldots, a_r=k_r, a_j \in [1,q] , j \geq r + 1
\]
\[
T_2: (a_1,a_2,a_3,\cdots,a_r,a_{r+1},\cdots), a_1=k_1,\ldots,a_r=k_r, a_{r+1} \in [2,q] , a_j \in [1,q], j \geq r + 2
\]
\[
T_3: (a_1,a_2,a_3,\cdots,a_r,a_{r+1},\cdots), a_1=k_1,\ldots,a_r=k_r, a_{r+1} \in [3,q], a_j \in [1,q], j \geq r + 2.
\]
\[
\ldots
\]
\[
T_{q-1}:  (a_1,a_2,a_3,\cdots,a_r,a_{r+1},\cdots), a_1=k_1,\ldots,a_r=k_r, a_{r+1} \in [q-1,q], a_j \in [1,q], j \geq r + 2.
\]
Here, $T_i$ represents the interval spanned by the maximum and minimum value of
$$\zeta^\star(a_1,a_2,a_3,\cdots,a_r,a_{r+1},\cdots).$$ Namely that $T_i$ is $$[\zeta^\star(k_1,k_2, \cdots, k_r, \{q\}^{\infty} ), \zeta^\star(k_1,k_2, \cdots, k_r, i, \{1\}^{\infty} ) ], $$ and the initial interval $D$ is of type $T_2$ with $r=0$. In this definition, we have $k_j\in [1,q], 1\leq j\leq r $, with the additional condition that $k_1 \geq 2$. Next, dividing an interval of type $T_i$ with arbitrary $1\leq i\leq q-1$ to an interval of type $T_1$ with $$ a_{r+1}=i,1\leq a_j\leq q,j\geq r+2 $$ and an interval of type $T_{ i+1}$ with  $$ i+1\leq a_{r+1}\leq q,1\leq a_j\leq q,j\geq r+2 $$ and the gap between these two intervals removed.

Hence, it suffices to verify that this subdivision of the interval $D$ satisfies Condition \ref{con} to obtain the desired conclusion $ D+D=[c_q,+\infty) $. In the following discussion, we view $\zeta^\star(2, \{1\}^{\infty})$ as $+\infty$.

If $1 \leq i \leq q-1$, we have
\[
T_i: N= \zeta^\star(k_1,k_2, \ldots, k_r, q, q, \ldots),M=\zeta^\star(k_1,k_2, \ldots, k_r, i, 1, 1, \ldots). 
\]
Then denote $L$ for the length of $T_i$.
\[
Subinterval \quad T_1:N= \zeta^\star(k_1,k_2, \ldots, k_r, i, q, q, \ldots),M=\zeta^\star(k_1,k_2, \ldots, k_r, i, 1, 1, \ldots).
\]
Then denote $L_1$ for the length of $T_1$.
\[
Subinterval \quad T_{i+1}:N= \zeta^\star(k_1,k_2, \ldots, k_r, q, q, \ldots),M=\zeta^\star(k_1,k_2, \ldots, k_r, i+1, 1, 1, \ldots).
\]
Then denote $L_{2} $ for the length of $T_{i+1}$.

Hence, the length of gap is $ L-L_1-L_2 $. We only need to prove $$ L-L_1-L_2 \leq L_1,\;\;\; L-L_1-L_2 \leq L_2 ,$$ then $ \eta(\mathcal{D}_q)+\eta(\mathcal{D}_q)=[c_q, +\infty) $ follows.

To prove $ L-L_1-L_2 \leq L_1 $, it  is equivalent to prove
\[
\begin{aligned}
\zeta^\star(k_1,k_2, \ldots, k_r,i, q, q, \ldots)&- \zeta^\star(k_1,k_2, \ldots, k_r, i+1, 1, 1, \ldots)
\\\leq \zeta^\star(k_1,k_2, \ldots, k_r, i, 1, 1, \ldots) &-\zeta^\star(k_1,k_2, \ldots, k_r, i, q, q, \ldots) 
\end{aligned}
\]
i.e.
\[
\begin{aligned}
&2\sum_{n_1 \geq \cdots \geq n_r \geq 1} \frac{1}{n_1^{k_1} \cdots n_r^{k_r}} \sum_{n_r \geq n_{r+1} \geq 1}\frac{1}{n^i_{r+1}} \prod_{n=2}^{n_{r+1}} \left(1 - \frac{1}{n^q}\right)^{-1}  
 \\ \leq &\sum_{n_1 \geq \cdots \geq n_r \geq 1} \frac{1}{n_1^{k_1} \cdots n_r^{k_r}} \sum_{n_r \geq n_{r+1} \geq 1} {n^ {1-i}_{r+1}} 
 +\sum_{n_1 \geq \cdots \geq n_r \geq 1} \frac{1}{n_1^{k_1} \cdots n_r^{k_r}} \sum_{n_r \geq n_{r+1} \geq 1} {n^ {-i}_{r+1}}
 \\ = & \sum_{n_1 \geq \cdots \geq n_r \geq 1} \frac{1}{n_1^{k_1} \cdots n_r^{k_r}} \sum_{n_r \geq n_{r+1} \geq 1} ( {n^ {1-i}_{r+1}} +{n^ {-i}_{r+1})}.
\end{aligned}
\]
Because of \[ 
\sum_{m\geq n_{r+1} \geq n_{r+2} \geq \cdots \geq 1}\frac{1}{n_{r+1} ^{q} n_{r+2}^{q} \cdots} =\prod_{n=2}^{m} \left(1-\frac{1}{n^q} \right)^{-1} 
\] 
for all $q,m\geq 1$. Here  we interpret that $\prod_{n=2}^{m} \left(1-\frac{1}{n^q} \right)^{-1}=1 $ for $m=1$.
Then it suffices to prove 
\[
2\cdot \prod_{n=2}^{m} \left(1-\frac{1}{n^q} \right)^{-1} \leq m+1 \tag{A}
\]
for all $q\geq 2, m\geq 1$.
Denote by  $F_m(q)=\prod_{n=2}^{m} \left(1-\frac{1}{n^q} \right)^{-1} $, it is easy to see that $ F_m(q) \leq F_m(2) $. 
For $m=1$, the statement $(A)$ is obvious. 
For $m\geq 2$, we have 
\[
2\cdot F_m(2)=2\cdot \prod_{n=2}^{m} \left(1-\frac{1}{n^2} \right)^{-1} = 2\cdot \prod_{n=2}^{m} \left(\frac{(n-1)(n+1)}{n^2} \right)^{-1}= \frac{4m}{m+1} \leq m+1.
\]
So we have proved $ L-L_1-L_2 \leq L_1 $.

Similarly, to prove $ L-L_1-L_2 \leq L_2 $, it is equivalent  to prove
\[
\begin{aligned}
&\zeta^\star(k_1,k_2, \ldots, k_r,i, q, q, \ldots) - \zeta^\star(k_1,k_2, \ldots, k_r, i+1, 1, 1, \ldots)
\\\leq \; &\zeta^\star(k_1,k_2, \ldots, k_r, i+1, 1, 1, \ldots) -\zeta^\star(k_1,k_2, \ldots, k_r, q, q, \ldots) 
\end{aligned}
\]
i.e.
\[
\begin{aligned}
&\sum_{n_1 \geq \cdots \geq n_r \geq 1} \frac{1}{n_1^{k_1} \cdots n_r^{k_r}} \sum_{n_r \geq n_{r+1} \geq 1}\frac{1}{n^i_{r+1}} \prod_{n=2}^{n_{r+1}} \left(1 - \frac{1}{n^q}\right)^{-1} 
\\+&\sum_{n_1 \geq \cdots \geq n_r \geq 1} \frac{1}{n_1^{k_1} \cdots n_r^{k_r}} \sum_{n_r \geq n_{r+1} \geq 1}\frac{1}{n^q_{r+1}} \prod_{n=2}^{n_{r+1}} \left(1 - \frac{1}{n^q}\right)^{-1}  
\\ \leq& 2\sum_{n_1 \geq \cdots \geq n_r \geq 1} \frac{1}{n_1^{k_1} \cdots n_r^{k_r}} \sum_{n_r \geq n_{r+1} \geq 1} {n^ {-i}_{r+1}}
\end{aligned}
\]
Then it suffices to prove 
\[
\sum_{m=1}^{n_r}(m^{-i}+m^{-q})\prod_{n=2}^m\left( 1-\frac{1}{n^q}\right) ^{-1} \leq 2\sum_{m=1}^{n_r}m^{-i}
\]
for all $n_r\geq 1$. 
It suffices to show that 
\[
F_m(q)=\prod_{n=2}^{m} \left(1-\frac{1}{n^q} \right)^{-1} \leq \frac{2m^{-i}}{m^{-i}+m^{-q}}
\]
for $m\geq 1$.
As we have
\[
F_m(2)=\prod_{n=2}^{m} \left(1-\frac{1}{n^2} \right)^{-1} = \frac{2m}{m+1} \leq \frac{2m}{m+m^{i+1-q}}=\frac{2m^{-i}}{m^{-i}+m^{-q}} 
\]
with $m\geq1,1\leq i\leq q-1$. Since  $ F_m(q) \leq F_m(2) $,  we have proved $$ L-L_1-L_2 \leq L_2 .$$

Finally, it suffices to prove that $$\eta(\mathcal{D}_q)\cdot \eta(\mathcal{D}_q)=[d_q, +\infty)$$ with 
$$d_q=\left(\zeta^\star(\{q\}^{\infty})\right)^2=   \prod_{n\geq 2}\frac{1}{\left(1-\frac{1}{n^q}\right)^2}.$$ Let $ D^*=[\frac{\mathrm{log}\; d_q}{2},+\infty) $. Define the set $D^*$ as the set formed by taking the logarithm of each element in $D$. 

As shown in Figure:

\begin{figure}[htbp]
	\centering
	
	\begin{tikzpicture}
		
		\def\start{0}          
		\def\r{3.5}            
		\def\s{2.6}          
		\def\t{3.8}            
		\def\b{\start+\r}      
		\def\c{\b+\s}          
		\def\d{\c+\t}          		
		
		\draw[thick] (\start,0) -- (\d,0);

		\foreach \x/\name in {\start/a, \b/b, \c/c, \d/d}
		{
			\draw (\x,0.1) -- (\x,-0.1);
		}
		{
		\node[below] at (\start,-0.15) {$\zeta^\star(k_1, \ldots, k_r, \{q\}^{\infty})$};  
		\node[below=0.6cm] at (\b,-0.15) {$ \zeta^\star(k_1, \ldots, k_r,i+1,  \{1\}^{\infty})$}; 
		\node[below] at (\c,-0.15) {$\zeta^\star(k_1, \ldots, k_r,i,  \{q\}^{\infty})$};      
		\node[below=0.6cm] at (\d,-0.15) {$\zeta^\star(k_1, \ldots, k_r,i, \{1\}^{\infty})$};
		}

		\draw[decorate,decoration={brace,amplitude=5pt,raise=4pt}]
		(\start,0) -- (\b,0) node[midway,above=8pt] {$r$};

		\draw[decorate,decoration={brace,amplitude=5pt,raise=4pt}]
		(\b,0) -- (\c,0) node[midway,above=8pt] {$s$};
	
		\draw[decorate,decoration={brace,amplitude=5pt,raise=4pt}]
		(\c,0) -- (\d,0) node[midway,above=8pt] {$t$};
		
	\end{tikzpicture}
	
	\caption{interval of $T_i$}
	\label{fig:interval}
\end{figure}

Instead of dealing with infinities, we first consider the product for $ \eta(\mathcal{D}_q) \cap  [1, \zeta^\star(2, \{1\}^{r-1} ) ] $ and then let $r \to +\infty$. As shown by the above figure, to show that 
\[
\mathrm{log}\,\left(    \eta(\mathcal{D}_q) \cap  [1, \zeta^\star(2, \{1\}^{r-1} ) ]     \right)+\mathrm{log}\, \left(    \eta(\mathcal{D}_q) \cap  [1, \zeta^\star(2, \{1\}^{r-1} ) ]     \right)\]
is a closed interval, by Hall's Condition \ref{con}, it suffices to show that 
\[
\begin{split}
&\;\;\;\;\,\mathrm{log}\, \left(  \zeta^\star(k_1, \ldots, k_r,i,  \{q\}^{\infty})  \right) - \mathrm{log}\, \left(   \zeta^\star(k_1, \ldots, k_r,i+1,  \{1\}^{\infty})  \right)\\
&\leq \mathrm{log}\, \left(    \zeta^\star(k_1, \ldots, k_r,i,  \{1\}^{\infty})   \right)-\mathrm{log}\, \left(  \zeta^\star(k_1, \ldots, k_r,i,  \{q\}^{\infty})  \right), 
\end{split}    \tag{A}
\]
\[
\begin{split}
&\;\;\;\;\,\mathrm{log}\, \left(   \zeta^\star(k_1, \ldots, k_r,i,  \{q\}^{\infty}) \right) - \mathrm{log}\, \left(    \zeta^\star(k_1, \ldots, k_r,i+1,  \{1\}^{\infty})   \right)\\
&\leq \mathrm{log}\, \left(  \zeta^\star(k_1, \ldots, k_r,i+1,  \{1\}^{\infty})    \right)-\mathrm{log}\, \left(   \zeta^\star(k_1, \ldots, k_r,  \{q\}^{\infty})  \right)  
\end{split}     \tag{B}
\]
for all $2\leq k_1\leq q, 1\leq k_2,\cdots, k_r\leq q, 1\leq i\leq q-1$ and 
\[
(k_1,\cdots, k_r)\neq (2,\{1\}^{r-1}).
\]
As 
\[
\zeta^\star(k_1,k_2, \ldots, k_r,i, \{1\}^{\infty})=
\sum_{n_1\geq \cdots \geq n_r \geq n_{r+1}\geq 1} \frac{1}{n_1^{k_1}\cdots n_r^{k_r} n_{r+1}^{i-1}}, \]
\[
\zeta^\star(k_1,k_2, \ldots, k_r,i+1, \{1\}^{\infty})=
\sum_{n_1\geq \cdots \geq n_r \geq n_{r+1}\geq 1} \frac{1}{n_1^{k_1}\cdots n_r^{k_r} n_{r+1}^{i}}, \]
\[
\zeta^\star(k_1,k_2, \ldots, k_r,i, \{q\}^{\infty})=
\sum_{n_1\geq \cdots \geq n_r \geq n_{r+1}\geq 1} \frac{1}{n_1^{k_1}\cdots n_r^{k_r} n_{r+1}^{i}}\cdot \prod_{m=2}^{n_{r+1}} \frac{1}{1-\frac{1}{m^q}   }, \]
\[
\zeta^\star(k_1,k_2, \ldots, k_r, \{q\}^{\infty})=
\sum_{n_1\geq \cdots \geq n_r \geq n_{r+1}\geq 1} \frac{1}{n_1^{k_1}\cdots n_r^{k_r} n_{r+1}^{q}}\cdot \prod_{m=2}^{n_{r+1}} \frac{1}{1-\frac{1}{m^q}   }, \]
 the inequalities $(A)$ and $(B)$ are  equivalent to
 \[
 \begin{split}
 &\;\;\;\;   \left(   \sum_{n_1\geq \cdots \geq n_r \geq n_{r+1}\geq 1} \frac{1}{n_1^{k_1}\cdots n_r^{k_r} n_{r+1}^{i}}\cdot \prod_{m=2}^{n_{r+1}} \frac{1}{1-\frac{1}{m^q}   }       \right)^2                   \\
 &\leq        \left(  \sum_{n_1\geq \cdots \geq n_r \geq n_{r+1}\geq 1} \frac{1}{n_1^{k_1}\cdots n_r^{k_r} n_{r+1}^{i-1}}     \right)\left(   \sum_{n_1\geq \cdots \geq n_r \geq n_{r+1}\geq 1} \frac{1}{n_1^{k_1}\cdots n_r^{k_r} n_{r+1}^{i}}   \right) ,       \\
 \end{split} \tag{${\mathrm{A}}^\prime$}
 \]
  \[
 \begin{split}
 &\;\;\left(   \sum_{n_1\geq \cdots \geq n_r \geq n_{r+1}\geq 1} \frac{1}{n_1^{k_1}\cdots n_r^{k_r} n_{r+1}^{i}}\cdot \prod_{m=2}^{n_{r+1}} \frac{1}{1-\frac{1}{m^q}   }        \right)\\
 &\cdot \left(  \sum_{n_1\geq \cdots \geq n_r \geq n_{r+1}\geq 1} \frac{1}{n_1^{k_1}\cdots n_r^{k_r} n_{r+1}^{q}}\cdot \prod_{m=2}^{n_{r+1}} \frac{1}{1-\frac{1}{m^q}   }   \right)\\
 &\leq   \left(     \sum_{n_1\geq \cdots \geq n_r \geq n_{r+1}\geq 1} \frac{1}{n_1^{k_1}\cdots n_r^{k_r} n_{r+1}^{i}}   \right)^2                  \\
  \end{split} \tag{${\mathrm{B}}^\prime$}
 \]
 for all $2\leq k_1\leq q, 1\leq k_2,\cdots, k_r\leq q, 1\leq i\leq q-1$ and 
\[
(k_1,\cdots, k_r)\neq (2,\{1\}^{r-1}).
\]
For $q\geq 2,n\geq 1$, one has
\[
\begin{split}
&\prod_{m=2}^n\frac{1}{1-\frac{1}{m^q}}\leq \prod_{m=2}^n\frac{1}{1-\frac{1}{m^2}}=\frac{2n}{n+1}.
\end{split}
\]
Thus it suffices to show that
\[
 \begin{split}
 &\;\;\;\;   \left(   \sum_{n_1\geq \cdots \geq n_r \geq n_{r+1}\geq 1} \frac{1}{n_1^{k_1}\cdots n_r^{k_r} n_{r+1}^{i}}\cdot     \frac{2n_{r+1}}{n_{r+1}+1}    \right)^2                   \\
 &\leq        \left(  \sum_{n_1\geq \cdots \geq n_r \geq n_{r+1}\geq 1} \frac{1}{n_1^{k_1}\cdots n_r^{k_r} n_{r+1}^{i-1}}     \right)\left(   \sum_{n_1\geq \cdots \geq n_r \geq n_{r+1}\geq 1} \frac{1}{n_1^{k_1}\cdots n_r^{k_r} n_{r+1}^{i}}   \right) ,       \\
 \end{split} \tag{C}
 \]
  \[
 \begin{split}
 &\;\;\left(   \sum_{n_1\geq \cdots \geq n_r \geq n_{r+1}\geq 1} \frac{1}{n_1^{k_1}\cdots n_r^{k_r} n_{r+1}^{i}}\cdot   \frac{2n_{r+1}}{n_{r+1}+1}    \right)\\
 &\cdot \left(  \sum_{n_1\geq \cdots \geq n_r \geq n_{r+1}\geq 1} \frac{1}{n_1^{k_1}\cdots n_r^{k_r} n_{r+1}^{q}}\cdot   \frac{2n_{r+1}}{n_{r+1}+1}    \right)\\
 &\leq   \left(     \sum_{n_1\geq \cdots \geq n_r \geq n_{r+1}\geq 1} \frac{1}{n_1^{k_1}\cdots n_r^{k_r} n_{r+1}^{i}}   \right)^2                  \\
  \end{split} \tag{D}
 \]
 for all $2\leq k_1\leq q, 1\leq k_2,\cdots, k_r\leq q, 1\leq i\leq q-1$ and 
\[
(k_1,\cdots, k_r)\neq (2,\{1\}^{r-1}).
\]

By the well-known Cauchy-Schwarz formula
\[
\big{(}\sum_i a_i b_i\big{)}^2\leq \big{(}\sum_i a_i^2\big{)} \big{(}\sum_i  b_i^2\big{)},
\]
we have
\[
\begin{split}
&\;\;\;\; \left(   \sum_{n_1\geq \cdots \geq n_r \geq n_{r+1}\geq 1} \frac{1}{n_1^{k_1}\cdots n_r^{k_r} n_{r+1}^{i}}\cdot     \frac{2n_{r+1}}{n_{r+1}+1}    \right)^2       \\
&=   \left(   \sum_{n_1\geq \cdots \geq n_r \geq n_{r+1}\geq 1} \left(\frac{n_{r+1}}{n_1^{k_1}\cdots n_r^{k_r} n_{r+1}^{i}} \right)^{\frac{1}{2}} \cdot  \left(\frac{n_{r+1}}{n_1^{k_1}\cdots n_r^{k_r} n_{r+1}^{i}} \right)^{\frac{1}{2}}    \frac{2}{n_{r+1}+1}    \right)^2   \\                                                                                        
&\leq    \left(  \sum_{n_1\geq \cdots \geq n_r \geq n_{r+1}\geq 1}      \frac{n_{r+1}}{n_1^{k_1}\cdots n_r^{k_r} n_{r+1}^{i}}                     \right)       \left(   \sum_{n_1\geq \cdots \geq n_r \geq n_{r+1}\geq 1}                             \frac{1}{n_1^{k_1}\cdots n_r^{k_r} n_{r+1}^{i}}\cdot \frac{4n_{r+1}}{(n_{r+1}+1)^2}  \right)                       \\
&\leq        \left(  \sum_{n_1\geq \cdots \geq n_r \geq n_{r+1}\geq 1} \frac{1}{n_1^{k_1}\cdots n_r^{k_r} n_{r+1}^{i-1}}     \right)\left(   \sum_{n_1\geq \cdots \geq n_r \geq n_{r+1}\geq 1} \frac{1}{n_1^{k_1}\cdots n_r^{k_r} n_{r+1}^{i}}   \right) .\\
\end{split}
\]
As a result, the inequality (C) is proved.

 Since $1\leq i\leq q-1$, one has
 \[
 \begin{split}
 &\;\;\left(   \sum_{n_1\geq \cdots \geq n_r \geq n_{r+1}\geq 1} \frac{1}{n_1^{k_1}\cdots n_r^{k_r} n_{r+1}^{i}}\cdot   \frac{2n_{r+1}}{n_{r+1}+1}    \right)\\
& \cdot \left(  \sum_{n_1\geq \cdots \geq n_r \geq n_{r+1}\geq 1} \frac{1}{n_1^{k_1}\cdots n_r^{k_r} n_{r+1}^{q}}\cdot   \frac{2n_{r+1}}{n_{r+1}+1}    \right)\\ 
&\leq  \left(   \sum_{n_1\geq \cdots \geq n_r \geq n_{r+1}\geq 1} \frac{1}{n_1^{k_1}\cdots n_r^{k_r} n_{r+1}^{i}}\cdot   \frac{2n_{r+1}}{n_{r+1}+1}    \right)\\
& \cdot \left(  \sum_{n_1\geq \cdots \geq n_r \geq n_{r+1}\geq 1} \frac{1}{n_1^{k_1}\cdots n_r^{k_r} n_{r+1}^{i+1}}\cdot   \frac{2n_{r+1}}{n_{r+1}+1}    \right)\\ 
&\leq \left[\frac{1}{2}\left(    \sum_{n_1\geq \cdots \geq n_r \geq n_{r+1}\geq 1} \frac{1}{n_1^{k_1}\cdots n_r^{k_r} n_{r+1}^{i}}\cdot   \frac{2n_{r+1}}{n_{r+1}+1} +    \sum_{n_1\geq \cdots \geq n_r \geq n_{r+1}\geq 1} \frac{1}{n_1^{k_1}\cdots n_r^{k_r} n_{r+1}^{i+1}}\cdot   \frac{2n_{r+1}}{n_{r+1}+1}     \right)         \right]^2\\
&\leq \left(    \sum_{n_1\geq \cdots \geq n_r \geq n_{r+1}\geq 1} \frac{1}{n_1^{k_1}\cdots n_r^{k_r} n_{r+1}^{i}}\right)^2.
\end{split}
 \]
 Here the second inequality follows from the following inequality:
 \[
 xy\leq \left(  \frac{x+y}{2}   \right)^2,\; \; x,y\in\mathbb{R}.
 \]
 As a result, the inequality (D) is proved.
Thus we have 
\[
\begin{split}
&\;\;\;\;\bigg{(} \eta(\mathcal{D}_q) \cap  [1, \zeta^\star(2, \{1\}^{r-1} ) ] \bigg{)}\cdot \bigg{(} \eta(\mathcal{D}_q) \cap  [1, \zeta^\star(2, \{1\}^{r-1} ) ] \bigg{)}\\
&=\big{[}\left(\zeta^\star(\{q\}^\infty)\right)^2, \left(\zeta^\star(2, \{1\}^{r-1})\right)^2\big{]}.  
\end{split}
\]
By letting $r\rightarrow +\infty$, one has 
\[\eta(\mathcal{D}_q)\cdot \eta(\mathcal{D}_q)=[\left(\zeta^\star(\{q\}^\infty)\right)^2 , +\infty).\]

For the proof of Theorem \ref{low}, $(iii),(iv)$, it suffices to note that if $A=B=\eta(\mathcal{D}_q)$ satisfies Hall's criterion (Theorem \ref{asu}), then the pair 
\[A= \eta(\mathcal{D}_q),B=-\eta(\mathcal{D}_q)\]
also satisfies Hall's criterion. Thus $\eta(\mathcal{D}_q)-\eta(\mathcal{D}_q)$ is a closed interval of $\mathbb{R}$.
Since $$\zeta^\star(2,\{1\}^{r})=\zeta^\star(2,\{1\}^{r-1},2,\{1\}^{\infty})=(r+1)\zeta(r+2)\in \eta(\mathcal{D}_q),\; r\geq 1,$$
one has $$\eta(\mathcal{D}_q)-\eta(\mathcal{D}_q)=\mathbb{R}.$$
By the same analysis, from the proof of Theorem \ref{low}, $(ii)$, we have 
$$\eta(\mathcal{D}_q)/\eta(\mathcal{D}_q)=(0,+\infty).$$

We have established that the infinite multiple zeta-star values, when all indices  do not exceed $q$, form a sumset $ \eta(\mathcal{D}_q)+\eta(\mathcal{D}_q) $ that is an interval. A natural question arises as to whether a similar conclusion holds when all indices are not less than $p$. That is, for $p\geq 2$, define $$ \mathcal{T}_p=\big{\{} (k_1,\cdots, k_r,\cdots)\in \mathcal{T}\;\big{|}\; k_i\geq p, i\geq 1 \big{\}} .$$ Since $\eta(\mathcal{T}_p)$ is not a closed set, we consider its closure $\overline{\eta(\mathcal{T}_p)}$. Unfortunately, some simple computations show that  $\overline{\eta(\mathcal{T}_p)} + \overline{\eta(\mathcal{T}_p)} $ is not a closed interval. 

When $p=2$, we have $\overline{\eta(\mathcal{T}_2)} \subset [1,2]$. In the first stage of the Cantor set subdivision, the intervals $$U_1 = [1 , \zeta^\star(3,\{2\}^{\infty})]=[1, 2\zeta(2)-2] $$ and $$U_2 = [\zeta^\star( 3, \{1\}^{\infty}), \zeta^\star(\{2\}^{\infty})] = [\zeta(2), 2] $$ are retained, while the interval $$U_{12} = ( \zeta^\star( 3, \{2\}^{\infty}),\zeta^\star( 3, \{1\}^{\infty})) = (2\zeta(2)-2, \zeta(2))$$ is deleted. Here 
\[
\zeta^\star (3, 2, 2, \dots) = \sum_{n_1 \geq n_2 \geq  \dots, \geq 1} \frac{1}{n_1^3 n_2^2 n_3^2 \cdots} = \sum_{n=1}^\infty \frac{1}{n^3} \cdot \frac{2n}{n+1} = 2 \zeta(2)-2
\]

See Figure:

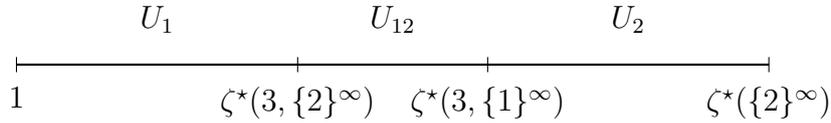
\begin{figure}[htbp]
	\centering
	\begin{tikzpicture}
	
		\def\start{0}         
		\def\r{3.7}           
		\def\s{2.5}            
		\def\t{3.7}          
		\def\b{\start+\r}    
		\def\c{\b+\s}         
		\def\d{\c+\t}

		\draw[thick] (\start,0) -- (\d,0);

		\foreach \x/\name in {\start/1, \b/ {\zeta^\star(3,\{2\}^{\infty})}, \c/{\zeta^\star( 3, \{1\}^{\infty})}, \d/{\zeta^\star(\{2\}^{\infty})} }
		{
			\draw (\x,0.1) -- (\x,-0.1);         
			\node[below] at (\x,-0.15) {$\name$}; 
		}

	\node at (\start+\r/2, 0) [above=8pt] {$U_1$};

	\node at (\b+\s/2, 0) [above=8pt] {$U_{12}$};

	\node at (\c+\t/2, 0) [above=8pt] {$U_2$};
		
	\end{tikzpicture}
	
	\caption{ first stage of $\overline{\eta(\mathcal{T}_2)}$ subdivision }
	\label{fig:interval2}
\end{figure}

So
\[
U_1 + U_1 =[ 2, 2 (2\zeta(2)-2) ]
\]
\[
U_1 + U_2 =[ 1+\zeta(2), 2\zeta(2)-2 +2 ]
\]
\[
U_2 + U_2 =[ 2\zeta(2), 4 ].
\]

Clearly, $ \overline{\eta(\mathcal{T}_2)} \subset U_1 \cup U_2 $, then $$\overline{\eta(\mathcal{T}_2)} + \overline{\eta(\mathcal{T}_2)} \subset (U_1 + U_1) \cup (U_1 + U_2) \cup (U_2 + U_2).$$ But $(U_1 + U_1) \cup (U_1 + U_2) \cup (U_2 + U_2)$ is not a closed interval because $ 2 (2\zeta(2)-2) <1+\zeta(2)$ and $ 2\zeta(2)-2 +2 \leq 2\zeta(2) $. As $2,4\in \overline{\eta(\mathcal{T}_2)}+\overline{\eta(\mathcal{T}_2)}$ and $$(U_1 + U_1) \cup (U_1 + U_2) \cup (U_2 + U_2)$$ is not a closed interval,  it follows that $\overline{\eta(\mathcal{T}_2)}+\overline{\eta(\mathcal{T}_2)}$ is not a closed interval. Furthermore, the sumset of $ \overline{\eta(\mathcal{T}_p)} \subset \overline{\eta(\mathcal{T}_2)} $ with itself cannot be a closed interval.

The same method applies to multiplication. Firstly, take the logarithm of each element in $U_1$, $U_2$ and $U_{12}$, then denote them respectively by $U_1^\star$, $U_2^\star$ and $U_{12}^\star$. We can get
\[
U_1^\star + U_1^\star =[ log(1), 2log(2\zeta(2)-2) ]
\]
\[
U_1^\star + U_2^\star =[ log( 1 \cdot \zeta(2) ),  log( ( 2\zeta(2)-2 ) \cdot 2) ]
\]
\[
U_2^\star + U_2^\star =[ 2log( \zeta(2) ), 2log2 ].
\]
Because of $ \frac{2}{3}\pi^2 - 4 < \frac{\pi^4}{36} $, we have $ log( 4( \zeta(2) - 1 ) ) <  2log( \zeta(2) ) $. So $(U_1^\star + U_1^\star) \cup (U_1^\star + U_2^\star) \cup (U_2^\star + U_2^\star)$ is not a closed interval. Thus 
\[
\overline{\eta(\mathcal{T}_2)} \cdot \overline{\eta(\mathcal{T}_2)}\]
is not a closed interval.
As 
\[
\overline{\eta(\mathcal{T}_p)} \subseteq \overline{\eta(\mathcal{T}_2)} \]
for any $p\geq 2$, 
$\overline{\eta(\mathcal{T}_p)} \cdot \overline{\eta(\mathcal{T}_p)}$ cannot be a closed interval for any $ p \geq 2 $.

\begin{rem}
	When thinking this problem, we may first consider the binary expansion $\tau$ of the function value. Define $L_m=\bigg{\{} (a_n)_{n\geq1}\;\bigg{|}\; a_n\in\mathbb{N}, a_n\geq m\bigg{\}}$. It is surprising to find that, by the Newhouse's Thickness Theorem \cite{new}, $ \overline{\tau(L_2)} + \overline{\tau(L_2)} $ forms a closed interval $[0,2/3]$. The key of proof is to compute the thickness $ t(\overline{\tau(L_2)}) $ of the Cantor set. 
	
\end{rem}

\begin{Thm}\label{ne}(Newhouse's thickness theorem)
	Let \(C_1\) and \(C_2\) be two linked Cantor sets, i.e. the size of largest gap of \(C_1\) is not greater than the diameter of \(C_2\) and vice versa. If \( t(C_1) \cdot t(C_2) \geq 1\), then the arithmetic sum
	\[
	C_1 + C_2 = \{x + y : x \in C_1, y \in C_2\}
	\]
	is an interval.
\end{Thm}

 \section*{Acknowledgements}
        The authors want to thank Prof. Hidekazu Furusho, Dr. Jianbo Sun,  Prof. Qingchun Tian and  Prof. Yufeng Wu for helpful comments.     This project is  supported by the National Natural Science Foundation of China (Grant No. 12571009) and the Natural Science Foundation of Hunan Province,
China (Grant No. 2026JJ40003).

\end{document}